\documentclass[10pt,reqno]{amsart}
\usepackage{amssymb}
\usepackage{amsmath}
\usepackage{enumitem}
\usepackage{graphicx}
\usepackage{mathrsfs}
\usepackage[all]{xy}
\usepackage{hyperref}
\usepackage{marginnote}
\usepackage[toc,page]{appendix}
\usepackage{color}

\usepackage[margin=1.18in]{geometry} 

\RequirePackage{mathrsfs} 
\usepackage[OT2,T1]{fontenc}
\DeclareSymbolFont{cyrletters}{OT2}{wncyr}{m}{n}
\DeclareMathSymbol{\Sha}{\mathalpha}{cyrletters}{"58}

\newtheorem{theorem}{Theorem}[section]
\newtheorem{lemma}[theorem]{Lemma}
\newtheorem{proposition}[theorem]{Proposition}
\newtheorem{corollary}[theorem]{Corollary}

\theoremstyle{definition}

\newtheorem{remark}[theorem]{Remark}

\numberwithin{equation}{section} \numberwithin{figure}{section}

\DeclareMathOperator{\Pic}{Pic} 
 \DeclareMathOperator{\NS}{NS}
 
\DeclareMathOperator{\Spec}{Spec}

\DeclareMathOperator{\Hom}{Hom}

\DeclareMathOperator{\Br}{Br}

\newcommand{\CH}{\operatorname{CH}}
\newcommand{\pr}{\operatorname{pr}}
\newcommand{\cl}{\operatorname{cl}}

\newcommand\FF{\mathbb{F}}
\newcommand\PP{\mathbb{P}}

\newcommand\ZZ{\mathbb{Z}}
\newcommand\NN{\mathbb{N}}
\newcommand\QQ{\mathbb{Q}}

\newcommand\CC{\mathbb{C}}
\newcommand\GG{\mathbb{G}}

\newcommand\Gm{\GG_\mathrm{m}}
\newcommand\OO{\mathcal{O}}

\newcommand{\et}{\textrm{\'{e}t}}

\newcommand{\dR}{\textrm{dR}}
\newcommand{\crys}{\textrm{crys}}
\newcommand{\fppf}{\textrm{fppf}}

\begin{document}

\title[Invariants of Fano varieties]{Invariants of   Fano varieties in families}

\author{Frank Gounelas}
\address{Frank Gounelas\\
	Humboldt Universit\"at Berlin \\
	Berlin\\
	Germany.}
\email{gounelas@mathematik.hu-berlin.de}

\author{Ariyan Javanpeykar}
\address{Ariyan Javanpeykar \\
	Institut f\"{u}r Mathematik\\
	Johannes Gutenberg-Universit\"{a}t Mainz\\
	Staudingerweg 9, 55099 Mainz\\
	Germany.}
\email{peykar@uni-mainz.de}



\begin{abstract} We show that the Picard rank is constant in families of Fano varieties (in arbitrary characteristic) and we moreover investigate the constancy of the index.  
\end{abstract}
\maketitle 

\section{Introduction}

The aim of this paper is to extend basic properties of invariants of complex algebraic (smooth) Fano varieties to
positive characteristic. 

For $X$ a smooth projective variety over a field $k$, the N\'eron-Severi group $\mathrm{NS}(X)$ is finitely generated.
We let $\rho(X)$ denote the rank of $\mathrm{NS}(X_{\bar{k}})$, where $\bar{k}$
is an algebraic closure of $k$. We refer to $\rho(X)$ as the (geometric) Picard rank of $X$. Recall that a smooth
projective geometrically connected variety $X$ over a field $k$ is a (smooth projective) Fano variety if the dual
$\omega_{X/k}^{\vee}$ of the canonical invertible sheaf is ample.  It is not hard to show that the Picard rank $\rho(X)$
of a Fano variety $X$ over $\mathbb C$ equals the second Betti number of $X$ \cite[Prop. 2.1.2]{IskPro}. In particular,
as Betti numbers are constant in smooth complex algebraic families, the Picard rank is constant in any complex algebraic
family of  Fano varieties parametrized by a connected variety. 

The constancy of the Picard rank in families plays an important role in the classification of Fano varieties. Motivated by the
classification problem for Fano varieties in positive characteristic, we show that the Picard rank is constant in
families of Fano varieties. 

Note that Fano varieties are rationally chain connected \cite{Campana, KMM}.  A proof of the following two results, using the decomposition of the diagonal, is well-known to experts; we include this proof in the appendix. 

\begin{theorem}\label{thm:picard_constant} 
    Let $X\to S $ be a smooth proper morphism of schemes whose geometric fibres are rationally chain connected smooth projective varieties. If $S$ is connected,
    then the geometric Picard rank is constant on the fibres, i.e., for all $s$ and $t$ in $S$, we have $$\rho(X_s) =
    \rho(X_t).$$
\end{theorem}

We were first led to investigate this problem when studying integral points on the complex algebraic stack of Fano varieties; see \cite[\S 2]{JL}. The results obtained  in \textit{loc.~ cit.} use deformation-theoretic techniques and only deal with characteristic zero or mixed characteristic. In fact, in \textit{loc.~cit.} Kodaira vanishing plays a crucial role. However, there are (smooth) Fano varieties in characteristic two which violate Kodaira vanishing \cite{LauritzenRao}. 

It is not difficult  to see that the Picard rank can jump in families of non-Fano varieties (e.g.\ K3 surfaces). However, note that the Picard rank is constant  in families of Enriques surfaces \cite{Liedtke}. (The constancy of the Picard rank in families of Enriques surfaces also follows from the results in \cite{Auel}.)

By the smooth proper base change theorem for \'etale cohomology, to prove Theorem \ref{thm:picard_constant}, it suffices
to show that the Picard rank equals the second $\ell$-adic Betti number for all prime numbers $\ell$ invertible in the
base field. As Fano varieties are rationally chain connected, the latter follows from the    next result.

\begin{theorem}\label{thm:tate}
	Let $X$ be a smooth projective rationally chain connected variety over an algebraically closed field $k$. For all prime numbers $\ell$, the homomorphism
		\[ \mathrm{NS}(X)\otimes\ZZ_\ell \to \mathrm{H}^2_{\mathrm{fppf}}(X, \ZZ_\ell(1)) \] is an isomorphism of $\ZZ_\ell$-modules.
\end{theorem}

For $X$ a scheme, we let $\Br(X)$ be the cohomological Brauer group $\mathrm{H}^2_{\et}(X,\mathbb G_m)$. Note
that, if $X$ is  a regular scheme, then $\Br(X)$ is a torsion abelian group \cite[Prop. 1.4]{GrothendieckBrauerII}. 
To prove Theorem \ref{thm:tate}, we use  geometric arguments following a suggestion of Jason Starr.   Namely, we use simple facts about Brauer groups of compact type curves (see Section \ref{section:compact_type}) to prove that the Brauer group of a rationally chain connected variety is killed by some integer $d\geq 1$ (see Proposition \ref{prop:Brauer_is_killed}). A well-known argument involving the Kummer sequence then concludes the proof of Theorem \ref{thm:tate} (and thus Theorem \ref{thm:picard_constant}).

Another natural invariant of a Fano variety is its  index.
Let $X$ be a Fano variety over a field $k$. We define its \emph{(geometric) index}
$r(X)$ to be the largest $r \in \NN$ such that $-K_{X_{\bar {k } } }$ is divisible by $r$ in $\Pic X_{\bar{k}}$. It is not hard to show that the index is constant in families of Fano varieties over the complex numbers.

It seems reasonable to suspect that the index is constant in families of Fano varieties in mixed or positive characteristic; see Proposition \ref{prop:index_almost_2} for some partial results. In general, we are only able to establish the constancy of the index up to multiplication by powers of the characteristic of the residue field. 

\begin{theorem}\label{thm:index_almost} 
Let $S$ be a trait  with generic point $\eta$. Assume that the characteristic $p$ of the closed point $s$ of $S$ is
positive. Let $X\to S$ be a smooth proper morphism of schemes whose geometric fibres are Fano varieties. Then, there is
an integer $i\geq 0$ such that $r(X_s) = p^i r(X_\eta)$.
\end{theorem}

In the hope of establishing the constancy of the index for families of Fano varieties, we investigate $p$-adic
``crystalline'' analogues of Theorem \ref{thm:tate}. For a perfect field $k$, we let $W = W(k)$ be the Witt ring of $k$. 

\begin{theorem}\label{thm:src}
    Let $X$ be a smooth projective connected scheme over an algebraically closed field $k$ of
    characteristic $p>0$. If $X$ is separably rationally connected or $X$ is a Fano variety with $\mathrm{H}^0(X,\Omega^1_X) =0$, then $\mathrm{H}^2_{\mathrm{crys}}(X/W)$ is torsion-free, $\mathrm{H}^1(X,\mathcal{O}_X)=0$, and the morphisms of $\mathbb Z_p$-modules 
        \[ \mathrm{NS}(X)\otimes \ZZ_p \to \mathrm{H}^2_{\mathrm{fppf}}(X,\mathbb Z_p(1))\to  \mathrm{H}^2_{\mathrm{crys}}(X/W)^{F-p}\] are isomorphisms.
\end{theorem}

It is currently not known whether $\mathrm{H}^1(X,\mathcal{O}_X)=0$, for $X$ a Fano variety over an algebraically closed field $k$.  
In \cite{ShepherdBarron} Shepherd-Barron proves that $\mathrm{H}^1(X,\mathcal{O}_X)=0$ for all Fano varieties $X$ of dimension at most three (see also \cite[Corollary~2]{Megyesi}). 

\subsection*{Acknowledgements} The authors would like to thank Jason Starr for sketching how an argument for proving Proposition \ref{prop:Brauer_is_killed}  should work. We are most grateful to Jean-Louis Colliot-Th\'el\`ene and Charles Vial for their comments and suggestions, and for pointing our a mistake in our proof of Theorem \ref{thm:appendix} in an earlier version. We also thank Ben Bakker, Francois
Charles, Cristian Gonz\'alez-Avil\'es, Ofer Gabber, Raju Krishnamoorthy, Christian Liedtke, Daniel Loughran, and Kay R\"ulling for
helpful discussions. The second named author gratefully acknowledges support from SFB/Transregio 45.

\section{Brauer groups of compact type curves}\label{section:compact_type} 
In this section, we let $k$ be a  field (of arbitrary characteristic). The aim of this section is to investigate Brauer groups of mildly singular curves over $k$.  Similar statements (assuming $k$ is of characteristic zero) are obtained by Harpaz-Skorobogatov \cite{HarpazSkoro}.

A proper   connected   reduced one-dimensional scheme over $k$ with only ordinary double points is a \textit{curve of compact type} if its dual graph is a tree. In other words, a nodal proper       connected reduced curve over $k$ is of compact type if and only if every node is disconnecting. 
Note that  the normalization of a compact type curve is the disjoint union of its irreducible components.

We will frequently use ``partial normalizations''. More precisely, let $s$ in $C(k)$ be a singular point of $C$. Since $s$ is a disconnecting node, there are precisely two closed subschemes $C_1$ and $C_2$ such that $s\in C_1(k)$, $s\in C_2(k)$, $C= C_1\cup C_2$, and $C_1 \cap C_2 = \{s\}$. Note that $C_1$ and $C_2$ are unique (up to renumbering). We define the \textit{partial normalization of $C$ with respect to $s$} to be the morphism $C'\to C$, where $C'= C_1 \sqcup C_2$. It is clear that the normalization $\widetilde{C} \to C$ of $C$ factors as $\widetilde{C}\to C'\to C$.
 
\begin{lemma}\label{lem:norm}
    Let $C$ be a compact type curve over $k$ and $Z$ an   integral scheme of finite type over $k$. Let
    $0\in C(k)$ be a singular point and let $C'\to C$ be the partial normalization with respect to $0$. Then the
    pullback morphism $\Br(Z\times C) \to \Br(Z\times C')$ is injective.
\end{lemma}
\begin{proof}
    Note that $C'$ is the disjoint union of two compact type curves, $C_1$ and $C_2$ say. Also, note that $0\in
    C(k)$ induces a point $ 0_1\in C_1(k)$ and a point $0_2\in C_2(k)$. Let $Y= Z\times C'$ and let $\nu: Y\to Z\times C$
    be the morphism induced by the partial normalization $C'\to C$. Note that $Y$ is the disjoint union of $Y_1:=Z\times C_1$ and
    $Y_2:=Z\times C_2$. Let $j:Z\cong Z\times \{0\}\to Z\times C$ be the closed immersion induced by $0\in C(k)$.  We have a
    short exact sequence of \'etale sheaves on $Z\times C$
    $$ 0 \to \GG_m \to \nu_*\GG_m \to j_*\GG_m\to 0. $$ 
    Since $\nu$ is a finite morphism, this induces an exact sequence
    $$ \  \Pic(Z\times C') \xrightarrow{f} \Pic Z \to \Br(Z\times C) \to \Br(Z\times C'). $$ 
    Note that, to prove the lemma, it suffices to show that $f$ is surjective. To do so, let $L$ be (the isomorphism
    class of) a line bundle on $Y$. Let $L_1$ be the induced line bundle on $Y_1$ and $L_2$ be the induced line bundle
    on $Y_2$. The homomorphism $f$ maps $L$  to the line bundle $p^\ast L_1 \otimes q^\ast L_2^{-1}$ on $Z$, where
    $p:Z\cong Z\times \{0_1\}\to Y_1$ is the section induced by $0_1$ and $q:Z\cong Z\times \{0_2\}\to Y_2$ is the
    section induced by $0_2$. It is clear that this map is surjective.
\end{proof}

\begin{corollary}\label{cor:norm_map}
    Let $C$ be a compact type curve  over $k$ and let $Z$ be a smooth   integral finite type
    scheme over $k$ with function field $L= K(Z)$. Then, the natural morphism $\Br(Z\times C)
    \to \Br(\Spec L\times_k C)$ is injective.
\end{corollary}
\begin{proof} 
    Let $N$ be the number of irreducible components of $C$. We argue by induction on $N$. If $N=1$, then $Z\times_k C  $ is   integral over $k$ and the natural pull-back morphism $\Br(Z\times_k C)\to
    \mathrm{Br}(K(Z\times_k C))$ is injective \cite[Cor IV.2.6]{milneetcoh}. In particular, as the natural morphism $\Spec K(Z\times_k
    C)\to Z\times_k C$ factors as $\Spec K(Z\times_k C)\to \Spec L\times_k C\to Z\times_k C$, it follows that $\Br(Z\times_k
    C)\to \Br(\Spec L\times_k C)$ is injective. 
	
    Thus, to prove the corollary, we may and do assume that $N\geq 2$. The natural morphism $\Spec L\to Z$ induces  the
    natural morphism $i:\Spec L \times C \to Z\times C$. Let $C'\to C$ be the partial normalization with respect to the
    choice of a singular point $0$ in $C(k)$. Note that there is a natural morphism $j:\Spec L\times_k C'\to Z\times_k
    C'$, and that $C' = C_1\sqcup C_2$, where $C_1$ and $C_2$ are compact type curves. Note that  the number of irreducible components of $C_1$ (resp.  $C_2$) is
    less than $N$.
	
    There are natural morphisms $f:Z\times_k C'\to Z\times_k C$ and $g:\Spec (L)\times_k C'\to \Spec (L)\times_k C$
    induced by the morphism $C'\to C$.  
    Note that $f^\ast$ is injective by Lemma \ref{lem:norm}. Since the number of irreducible components of $C_1$ (resp. $C_2)$ is less than $N$, by the
    induction hypothesis, the natural pull-back morphisms $j_1^\ast:\Br(Z\times C_1) \to \Br(\Spec L \times C_1)$ and
    $j_2^\ast:\Br(Z\times C_2) \to \Br(\Spec L \times C_2)$ are injective. Note that $$\Br(Z\times C') = \Br(Z\times
    C_1)\oplus \Br(Z\times C_2)$$ and likewise $$\Br(\Spec L\times C') = \Br(\Spec L \times C_1)\oplus \Br(\Spec L\times
    C_2).$$ Thus, we see that  $j^\ast = j_1^\ast \oplus j_2^\ast$ is injective. It follows that $ j^\ast \circ
    f^\ast$ is injective. As $g^\ast\circ i^\ast = j^\ast \circ f^\ast$,  we conclude that $i^\ast$ is injective.
\end{proof}
 
A compact type curve $C$ over $k$ is of genus zero if  all   irreducible components of $C_{\bar{k}}$ have genus zero.  Note that, in this case, all irreducible components of $C_{\bar{k}}$ are isomorphic to $\mathbb P^1_{\bar{k}}$ over $\bar{k}$.

\begin{lemma}\label{lem:manin}
    Let $C$ be a compact type curve of genus zero over $k$. Assume that all irreducible components of $C$ are geometrically irreducible, and that $C(k)\neq \emptyset$. Then the natural pull-back
    morphism $\mathrm{Br}(k) \to \mathrm{Br}(C)$ is an isomorphism.
\end{lemma}
\begin{proof} 
    Let $N$ be the number of irreducible components of $C$. We argue by induction on $N$. Firstly, if $N=1$, then the
    result is well-known. Indeed, if $N=1$, then $C\cong \mathbb P^1_k$. Let $k\to k^s$ be a separable closure of $k$. It is not hard to show that $\Br(k) =  \ker[\Br(\PP^1_k)\to \Br(\PP^1_{k^s})]$; see   \cite[Thm.~42.8]{manin}.  By a theorem of Grothendieck \cite[Corollaire~5.8]{GrothendieckBrauerIII} (see also \cite{Saltman}), we have that $\Br(\PP^1_{k^s}) =0$. We conclude that the result holds for $N=1$. Thus, to prove the lemma, we may and do assume that
    $N\geq 2$.
 	
    Let $s:\Spec k\to C$ be a singular point. Let $\nu:C'\to C$ be the partial normalization of $C$ with respect to
    $s$. Note that $C' = C_1 \sqcup C_2$, where $C_1$ and $C_2$ are compact type curves of genus zero over $k$.  Note that the irreducible components of $C_1$ and $C_2$ are geometrically irreducible, and that $C_1(k) \neq\emptyset$ and $C_2(k)\neq \emptyset$. Let
    $s_1:\Spec k\to C_1$ and $s_2:\Spec k\to C_2$ be the sections corresponding to $s:\Spec k \to C$.  Note that the
    following diagram of groups 
    \[ 
    \xymatrix{
    \Br(C) \ar[d]_{s^*}\ar[rr]^{\nu^*} & & \Br(C_1) \oplus \Br(C_2) \ar[d]^{s_1^*\oplus s_2^*} \\
    \Br(k) \ar[rr]_{\Delta} & &  \Br(k) \oplus \Br(k)	
    }
    \] 
    is commutative.  It follows from Lemma \ref{lem:norm} that the morphism $\nu^\ast$ is injective. By the induction
    hypothesis, since the number of irreducible components of $C_1$ and $C_2$ is less than $N$, the natural pull-back
    morphisms $\Br(k) \to \Br(C_1)$ and $\Br(k)\to \Br(C_2)$ are isomorphisms. In particular, $s_1^*\oplus s_2^*:
    \Br(C_1)\oplus \Br(C_2) \to \Br(k)\oplus \Br(k)$ is an isomorphism. We conclude that $(s_1^*\oplus s_2^*)\circ
    \nu^\ast$ is injective. However, since $\Delta^\ast \circ s^* = (s_1^*\oplus s_2^*) \circ \nu^*$, we conclude that
    $s^*$ is injective. Since $s^*:\Br(C) \to \Br(k)$ is surjective, we conclude that $s^*:\Br(C) \to \Br(k)$ is an
    isomorphism. The result follows.
\end{proof}

\begin{remark}
    Let $E$ be an elliptic curve over $\mathbb Q$ with $\mathrm{H}^1_{\et}(\Spec \mathbb{Q},E)\neq 0$. As is shown
    in \cite[\S 2]{Lichtenbaum}, the Leray spectral sequence induces  a short exact sequence 
    \[ \Br(\mathbb{Q})\to \Br(E) \to \mathrm{H}^1_{\et}(\Spec \mathbb{Q},E)\to 0.\] 
    Since $E(\QQ) \neq \emptyset$, the natural pull-back morphism $\Br(\mathbb{Q})\to \Br(E)$ is injective. As
    $\mathrm{H}^1_{\et}(\Spec \mathbb{Q}, E) $ is non-trivial, the natural morphism $\Br(\mathbb{Q})\to \Br(E)$ is not
    an isomorphism. 
\end{remark}

\begin{corollary}\label{cor:product}
    Let  $Z$ be a smooth   integral scheme of finite type over $k$. Let $C$ be a compact type
    curve of genus zero over $k$ with $C(k)\neq \emptyset$ and whose irreducible components are geometrically irreducible. Then  the natural pull-back morphism $\Br(Z) \to \Br(Z \times
    C)$ is an isomorphism.
\end{corollary}
\begin{proof}
    Let $0\in C(k)$ be a rational point and $s:Z\to Z\times C$ the corresponding section to the projection morphism
    $p:Z\times C\to Z$. Consider the induced maps on Brauer groups $p^*:\Br(Z)\to\Br(Z\times C)$ and
    $s^*:\Br(Z\times C)\to \Br(Z)$. Since the composition $s^*\circ p^*$ is   the identity, it follows that $s^*$
    is surjective.  
 	
    Let $L=k(Z)$ be the function field of $Z$. Let $i:\Spec L \to Z$ be the natural morphism. Let $j:\Spec L \times C\to Z
    \times C$ be the corresponding natural morphism. Note that $s\circ i = j \circ t$, where $t:\Spec L \to \Spec L
    \times C$ is induced by the point $0\in   C(k)$. In particular, since $s\circ i = j\circ t$, we have $i^\ast \circ s^\ast = t^\ast \circ j^\ast$. 
    By Corollary \ref{cor:norm_map}, the morphism $j^*$ is injective.  Moreover, by  Lemma \ref{lem:manin}, the morphism $t^*$ is an isomorphism (hence injective). In particular,
    the composition $t^\ast\circ j^\ast$ is injective. However, since $i^\ast \circ s^\ast  =
    t^\ast \circ j^\ast$,  it follows that $s^*$ is injective. We conclude that $s^*$ is an isomorphism, and that $p^*$
    is its inverse.
\end{proof}
 
\begin{remark}
    Corollary \ref{cor:product} fails without the assumption that $C(k)\neq \emptyset$. Indeed, if $C$ is a smooth proper
    curve of genus zero over a field $k$, then $\Br(k)\to \Br(C)$ is injective if and only if $C$ has a rational point.
\end{remark}

\section{The cycle class map}\label{section:cycle_class_map}
   Let $X$   be a smooth projective connected scheme over an algebraically closed field $k$. For any prime $\ell$   we have the
   Kummer sequence of fppf sheaves $$ 0\to \mu_{\ell^n} \to \Gm \to \Gm \to 0.$$ Using Grothendieck's theorem that for
   smooth sheaves the fppf and \'etale cohomology agree, we see that the long exact sequence in cohomology induces $$
   0\to \Pic(X)/\ell^n \to \mathrm{H}^2_\fppf(X, \mu_{\ell^n}) \to \Br(X)[\ell^n] \to 0$$ where the last term is the
   $\ell^n$-torsion in the cohomological Brauer group $\Br (X) = \mathrm{H}^2_\et(X, \Gm)$. Since $\Pic^0(X)$ is the group of closed points on an abelian
   variety over an algebraically closed field, it is a divisible group. In particular, the equality $\Pic(X)/\Pic^0(X) = \NS(X)$ induces $\Pic(X)/\ell^n
   = \NS(X)/\ell^n$. Passing to the projective limit we obtain $$ 0 \to \Pic(X)\otimes\ZZ_\ell \to \mathrm{H}^2_\fppf(X,
   \ZZ_\ell(1)) \to T_\ell\Br(X) \to 0; $$  see \cite[(5.8.5)]{Illusie}.
   
   \begin{lemma}\label{lem:tate_criterion}
 If there is an integer
   	$d\geq 1$ such that $\mathrm{Br}(X)$ is killed by $d$, then for all prime numbers $\ell$, the homomorphism
   	$ \mathrm{NS}(X)\otimes\ZZ_\ell \to \mathrm{H}^2_{\mathrm{fppf}}(X, \ZZ_\ell(1)) $ is an isomorphism of $\ZZ_\ell$-modules.
   \end{lemma} 
   \begin{proof}  
   	Our assumption on the Brauer group of $X$ implies that, for all prime numbers $\ell$, the $\ell$-adic Tate module $T_\ell
   	\mathrm{Br}(X)$ is zero. The statement therefore follows immediately from the discussion above.
   \end{proof}

  Note that, for a prime $\ell\in k^\ast$, we have that  $\mathrm{H}^2_{\fppf}(X,\mathbb Z_\ell(1)) =
   \mathrm{H}^2_{\et}(X,\mathbb{Z}_\ell(1))$. We will often use this.
   
   \begin{remark}\label{remark:int_Tate}
   	Let $X$ be a smooth projective geometrically connected scheme over a finitely generated field $K$, let $\bar{K}$ be an algebraic closure of $K$, and let $\ell$ be a prime number with $\ell\in K^\times$. Assume that $\Br(X)$ is killed by some positive integer $d$. Then, by Lemma \ref{lem:tate_criterion}, the natural map
   	\[\NS(X_{\bar{K}}) \otimes \ZZ_\ell \to \mathrm{H}^2_{\et}(X_{\bar{K}},\ZZ_\ell)\] is an isomorphism of $\ZZ_\ell$-modules compatible with the action of the absolute Galois group $G=\mathrm{Gal}(\bar{K}/K)$. If we   take    $G$-invariants on both sides of this isomorphism we obtain an integral version of the Tate conjecture  for divisors on  $X$. (Note that $\NS(X)$ is not necessarily isomorphic to the subgroup of $G$-invariants in $\NS(X_{\bar{K}})$. In particular, the ``naive integral'' analogue of   Tate's conjecture  for divisors fails for such $X$ over $K$; see for instance \cite{Schoen}.) 
   \end{remark}

\begin{corollary}\label{cor:crys} Assume that the characteristic $p$ of $k$ is positive.
	Let $W$ be the Witt ring of $k$ and let $K$ be the fraction field of $W$. If there is an integer
	$d\geq 1$ such that $\mathrm{Br}(X)$ is killed by $d$, then  the $K$-linear morphism
	\[ \mathrm{NS}(X)\otimes K \to \mathrm{H}^2_{\mathrm{crys}}(X/W) \otimes K \] is an isomorphism of $K$-vector spaces.
\end{corollary}
\begin{proof}  
	Note that the homomorphism $ \mathrm{NS}(X)\otimes K \to \mathrm{H}^2_\crys(X/W) \otimes K $ is injective \cite[Remarque~6.8.5]{Illusie}. Let $\ell$ be a prime number with $\ell \in k^\ast$. By Lemma \ref{lem:tate_criterion}, we have $$\dim_K \mathrm{NS}(X) \otimes_\ZZ K= \dim_K \mathrm{H}^2_{\fppf}(X,\mathbb Q_\ell) \otimes_{\QQ_\ell} K = \dim_K \mathrm{H}^2_{\et}(X,\mathbb Q_\ell) \otimes_{\QQ_\ell} K.$$ Moreover,  by \cite{KatzMessing} (see also \cite[1.3.1]{Illusie2}), we have that $$\dim_K \mathrm{H}^2_{\et}(X,\mathbb Q_\ell) \otimes_{\QQ_\ell} K = \dim_K \mathrm{H}^2_\crys(X/W) \otimes K.$$ This concludes the proof (as any injective $K$-linear map of finite-dimensional $K$-vector spaces of equal dimension is an isomorphism).
\end{proof}


 \section{Brauer groups of rationally chain connected varieties}
 In this section we prove Theorems \ref{thm:picard_constant} and \ref{thm:tate}.
 We let $k$ denote an algebraically closed field.
 
 \begin{lemma}\label{brauer-restriction-corestriction-singular} 
 	Let $X$ and $Y$ be quasi-projective   connected   reduced schemes over  $k$. Let $f:X\to Y$ be
 	a generically finite dominant morphism. If $Y$ is smooth, then there is an integer $d\geq 1$ such that
 	the kernel of the induced map $f^*:\Br (Y)\to \Br (X)$ is killed by $d$.
 \end{lemma}
 
 \begin{proof}
 	Let $X'\subset X$ be an open subscheme such that $X'$ is a    smooth  quasi-projective geometrically integral scheme over $k$ and $X'\to Y$ is dominant.  In particular, $X'\to Y$ is a generically finite dominant morphism of smooth geometrically integral quasi-projective schemes over $k$. Therefore, by \cite[Prop.~1.1]{ISZ} (which holds over any field), the kernel of the natural pull-back morphism $\Br(Y)\to
 	\Br(X')$ is killed by  the generic degree of $X'\to Y$. Since
 	the kernel of the morphism $\Br(Y)\to \Br(X)$ is contained in the kernel of $\Br(Y)\to \Br(X')$, the result
 	follows.
 \end{proof}
 

 \begin{proposition}\label{prop:Brauer_is_killed}
 	Let $X$ be a smooth projective  geometrically connected scheme over  $k$. If $X$ is rationally chain
 	connected, then there is an integer $d\geq 1$ such that $\mathrm{Br}(X)$ of $X$ is killed by $d$.
 \end{proposition}
 \begin{proof}  Fix   a general point $x$ in $X$.
 	Since $X$ is rationally chain connected,  there exist a smooth geometrically integral variety $Z$ with $\dim (Z)= \dim (X) - 1$,  a compact type curve $T$ of genus zero over $k$, and a dominant generically finite morphism $F:Z\times T\to X$ such that, for all $z$ in $Z$, the image of $F_z: T \cong \{z\}\times T \to X$ contains $x$. Note that, replacing $Z$ by a dense open if necessary, there is a section $s:Z\to Z\times T$  such that the composition $$ Z \xrightarrow{s}
 	Z\times T \xrightarrow{F} X $$ contracts $Z$ to $x$. 
 	In particular, as the pullback morphism $(F\circ s)^*: \Br (X) \to \Br (Z)$ factors through $\Br(k) =0$,  we see that   $(F\circ s)^\ast$ 
 	is the zero map. Now, since the natural pull-back morphism  $\Br(Z)\to \Br(Z\times T)$ is an isomorphism (Corollary \ref{cor:product}), we see that $\Br (Z) \to \Br (Z\times T)$ is the zero map.
 	However, by 
 	Lemma \ref{brauer-restriction-corestriction-singular}, there is an integer $d\geq 1$ such that the kernel of
 	$\mathrm{Br}(X) \to \mathrm{Br}(Z\times T)$ is killed by $d$. As the kernel of $\mathrm{Br}(X)\to
 	\mathrm{Br}(Z\times T)$ equals $\mathrm{Br}(X)$, we conclude that $\mathrm{Br}(X)$ is killed by $d$ as required.
 \end{proof}

 \begin{proof}[Proof of Theorem \ref{thm:tate}]
 	By Proposition \ref{prop:Brauer_is_killed}, the Brauer group of a rationally chain connected variety over an algebraically closed field is killed by some integer $d\geq 1$. In particular, the  theorem follows from    Lemma \ref{lem:tate_criterion}. 
         \end{proof}
 
 \begin{remark} Let $X$ be a smooth projective rationally chain connected variety over the algebraic closure of a finite field.  Then, it follows from  \cite[Theorem 0.4.(b)]{milne} and the fact that $\Br(X)$ is killed by some integer $d\geq 1$ (Proposition \ref{prop:Brauer_is_killed}) that the  	Brauer group of $X$ is finite.
 \end{remark}
 
 
 \begin{remark} Combining Remark \ref{remark:int_Tate} with 
 	Theorem \ref{thm:tate}, we obtain an integral analogue of the Tate conjecture for divisors on rationally connected smooth projective varieties.
 \end{remark}
 
 \begin{proof}[Proof of Theorem \ref{thm:picard_constant}]  Let $s$ and $t$ be points in $S$.
 	Let $\ell$ be a prime number such that $\ell \neq \mathrm{char}(k(s))$ and $\ell \neq \mathrm{char}(k(t))$. By the
 	smooth proper base-change theorem, $\dim_{\QQ_\ell} \mathrm{H}^2_{\et}(X_{\overline{s}}, \QQ_\ell) = \dim_{\QQ_\ell}
 	\mathrm{H}^2_{\et}(X_{\overline{t}},\QQ_\ell)$. In particular, since \'etale and fppf cohomology with $\mathbb Q_\ell$-coefficients agree,  Theorem
 	\ref{thm:tate} implies that the Picard ranks of $X_s$ and $X_t$ are equal.  
 \end{proof}
 

\section{Crystalline cohomology of separably rationally connected varieties}

We now prove a $p$-adic analogue of Theorem \ref{thm:tate} for separably rationally connected varieties; we refer the reader to \cite[IV.3]{kollar} for the basic definitions.


\begin{proposition}\label{prop:vanishing}
	Let $X$ be a smooth projective connected scheme over an algebraically closed field $k$ of characteristic $p\geq 0$. 
        \begin{enumerate}
            \item If $X$ is separably rationally connected, then $\mathrm{H}^0(X,\Omega^1_X) = \mathrm{H}^1(X,\mathcal
            O_X)=0$. If, in addition, $p>0$, then $\mathrm{H}^i_{\mathrm{crys}}(X/W)$ is torsion-free for $i\leq2$. 
            \item If $X$ is a Fano variety with $\mathrm{H}^0(X,\Omega^1_X) =0$, then $\mathrm{H}^1(X,\mathcal O_X)=0$.
            If, in addition, $p>0$, then $\mathrm{H}^i_{\mathrm{crys}}(X/W)$ is torsion-free for $i\leq2$.
	\end{enumerate}
\end{proposition}
\begin{proof} 
    We may and do assume that $p>0$. If $X$ is separably rationally connected, then it follows from
    \cite[IV.3.8]{kollar} that $\mathrm{H}^0(X,\Omega^1_X) =0$ which we may now assume to prove both statements of the
    proposition. Now, in the separably rationally connected case, the result about $\mathrm{H}^1(X,\OO_X)$ is contained in
    \cite{gounelas}; we extend the arguments to the Fano case as well. Since $\pi_1^{\et}(X)$ is finite
    \cite{chambertloirpi1} and the $p$-power torsion in $\pi_1^{\et}(X)$ is trivial by a result of Esnault
    \cite[Proposition 8.4]{chambertloirbourbaki}, we see that $\mathrm{H}^1_{\et}(X,\FF_p) =
    \Hom(\pi_1^{\et}(X),\FF_p)=0$. Note that $\mathrm{H}^1(X,\FF_p)\otimes k$ is the semi-simple part of the action of
    Frobenius on $\mathrm{H}^1(X,\mathcal{O}_X)$. Thus, Frobenius is nilpotent on $\mathrm{H}^1(X,\OO_X)$. However,
    Frobenius is also injective on $\mathrm{H}^1(X,\mathcal{O}_X)$, as its kernel is $\mathrm{H}^0(X, \OO_X/F\OO_X)\subset
    \mathrm{H}^0(X,\Omega^1_X)=0$. Since $F$ is both injective and nilpotent on $\mathrm{H}^1(X,\mathcal{O}_X)$, the group $\mathrm{H}^1(X,\mathcal{O}_X)$ has to be zero.
	
    To conclude the proof, one can now apply \cite[Thm.~II.5.16]{Illusie} or argue more directly as follows. By \cite[1.3.7]{Illusie2},  there is a universal coefficient exact sequence $$ 0\to \mathrm{H}^1_\crys(X/W)\otimes k \to \mathrm{H}^1_\dR(X/k) \to
    \mathrm{Tor}_1^W(\mathrm{H}^2_\crys(X/W),k) \to 0. $$  The existence of the Fr\"ohlicher spectral sequence shows that
    $\mathrm{h}^{0,1}(X)+\mathrm{h}^{1,0}(X)\geq \mathrm{h}^1_\dR(X/k).$ As $\mathrm{h}^{0,1}(X) = \mathrm{h}^{1,0}(X) =0$, it follows that $\mathrm{h}^1_\dR(X/k) =0$.
    Thus, $$\mathrm{Tor}_1^W(\mathrm{H}^2_\crys(X/W),k) =0,$$ thereby showing that $\mathrm{H}^2_\crys(X/W)$ is torsion-free. That
    $\mathrm{H}^1_\crys(X/W)$ is torsion free for any smooth projective variety is a standard result.
\end{proof}

\begin{proof}[Proof of Theorem \ref{thm:src}] 
    By Proposition \ref{prop:vanishing}, $\mathrm{H}^2_{\crys}(X/W)$ is torsion-free and $\mathrm{H}^0(X,\Omega^1_X) =
    \mathrm{H}^0(Z_1\Omega^1_Z) =0$, so that the result follows from \cite[Thm.~II.5.14]{Illusie} and Theorem \ref{thm:tate}. 
\end{proof}

\section{The index in families of Fano varieties}

We first show that the (geometric) index $r(X)$ of a Fano variety $X$  is bounded by $\dim(X)+ 1$.
\begin{lemma}\label{lem:indexdimension}
	Let $X$ be a  Fano variety over a  field $k$. Then  $r(X)\leq \dim(X)+1$.
\end{lemma}
\begin{proof} We may and do assume that $k$ is algebraically closed. Let $n:=\dim X$. 
	By \cite[V.1.6.1]{kollar}, there is a rational curve $C$ so that $-K_X.C\leq n+1$. Writing $-K_X=r(X)H$, we get $r(X)\leq r(X)H.C=-K_X.C\leq n+1$.
\end{proof}

\begin{proposition}\label{prop:constancy_of_p_index}
    Let $S$ be a trait with generic point $\eta$ and closed point $s$. Let $p = \mathrm{char}(k(s))$. Let $f:X\to S$
    be a smooth proper morphism of schemes whose geometric fibres are Fano varieties. 
    \begin{enumerate}
        \item If $p=0$, then $r(X_s) = r(X_\eta)$. 
        \item If $p>0$, then there is an integer $i\geq 0$ such that
        $r(X_s) = p^{i} r(X_\eta)$. 
        \item If  $r(X_s)>r(X_\eta)$, then   
        $0<p\leq (n+1)/r(X_\eta)$.
    \end{enumerate}
\end{proposition}
\begin{proof} We may and do assume that $k(s)$ is algebraically closed. 
    Let $X_{\overline{\eta}}$ denote the geometric generic fibre of $X\to S$.  We may and do assume that
    $\NS(X_\eta)  = \NS(X_{\overline{\eta}})$. 

    Firstly, since $S$ is regular integral and noetherian, the homomorphism  $\Pic(X) \to \Pic(X_\eta)$ is an
    isomorphism \cite[Lem.~3.1.1]{Har94}. 

   For all $\ell \neq p$, the natural morphism $\NS(X_{\eta})\otimes \ZZ_\ell = \NS(X_{\overline{\eta}})\otimes \ZZ_\ell \to
    \mathrm{H}^2_{\et}(X_{ \overline{\eta}},\ZZ_\ell)$ is an isomorphism (Theorem \ref{thm:tate}). Similarly, for all
    $\ell \neq p$, the natural morphism  $\NS(X_{s})\otimes \ZZ_\ell \to \mathrm{H}^2_{\et}(X_{ {s}},\ZZ_\ell)$ is an
    isomorphism (Theorem \ref{thm:tate}). By smooth proper base change for $\ell$-adic cohomology, for all $\ell \neq
    p$,  there is a natural $\mathbb Z_\ell$-linear isomorphism $\mathrm{H}^2_{\et}(X_{ \overline{\eta}},\ZZ_\ell)
    =\mathrm{H}^2_{\et}(X_{ {s}},\ZZ_\ell)$.  This proves $(1)$ and $(2)$.
    
    Finally, to prove $(3)$, by $(1)$, we may and do assume that $p>0$.  Also, by our assumption and $(2)$, we have that $r(X_s) = p^i r(X_\eta)$ with $i\geq 1$. Therefore, by  Lemma \ref{lem:indexdimension},  we immediately get $p\leq p^i = r(X_s)/r(X_\eta) \leq 
    (n+1)/r(X_\eta)$.
 \end{proof}

\begin{proof}[Proof of Theorem \ref{thm:index_almost}]
	This follows from Proposition \ref{prop:constancy_of_p_index}.
\end{proof}

\begin{proposition}\label{prop:index_almost_2}
    Let $S$ be a trait with generic point $\eta$ and closed point $s$. Let $p = \mathrm{char}(k(s))$. Let $f:X\to S$
    be a smooth proper morphism whose geometric fibres are Fano varieties. Assume any one of the following conditions:
    \begin{enumerate}
        \item $p>\dim X_s + 1$, or
        \item $\mathrm{H}^2(X_s, \OO_{X_s})=0$, or
        \item  there exists an integral $1$-cycle $T$ on $X_{\bar{\eta}}$ such that $-K_{X_{\bar{\eta}}}. T=r(X_{\bar{\eta}})$.
    \end{enumerate}
    Then, the index is constant on the fibres of $f:X\to S$, i.e.\ $r(X_\eta)=r(X_s)$. 
\end{proposition} 

\begin{remark} 
    Koll\'ar asked in \cite[V.1.13]{kollar} whether there exists a rational curve $T$ (instead of an integral sum of
    $1$-cycles) which satisfies the third condition of the proposition.  
\end{remark}

\begin{proof}[Proof of Proposition \ref{prop:index_almost_2}] 
    We may assume that $k(s)$ is algebraically closed. Assume that $p>\dim X_s+1$. Then, the result follows from $(3)$
    in Proposition \ref{prop:constancy_of_p_index}.
    
    Under the second assumption, assume that $r(X_s) = p^ir(X_\eta)$ with $i>0$. Then there exists an ample divisor
    $H_0\in\Pic(X_s)$ so that $-K_{X_s}=p^ir(X_\eta)H_0$. But the assumption $\mathrm{H}^2(X_s, \OO_{X_s})=0$ implies
    that all the obstructions to lifting a line bundle to the formal completion vanish, so $H_0$ lifts to a divisor
    $H\in\Pic(X_\eta)$. This gives a contradiction to the definition of the index $r(X_\eta)$.
 
    (The final assumption was suggested to us by Jason Starr.) Firstly, replacing $S$ by a finite flat cover $S'\to S$ with $S'$ a trait if necessary, we may and do assume that $T$ is descends to $X_{\eta}$, and that $\NS(X_{\eta}) = \NS(X_{\bar{\eta}})$. By abuse of notation, let $T$ be an integral $1$-cycle on $X_\eta$ such that  $-K_{X_\eta}.T=rH.T=r$, where $H$ is a
    primitive element of $\NS(X_\eta)$. Now, if $T=\sum m_i C_i$ for integral curves $C_i$ on $X_\eta$, then the
    curves $C_i$ all specialise to (possibly reducible) curves $\bar{C}_i$ in the special fibre $X_s$ but the
    intersection number $r=-K_{X_\eta}.T = -K_{X_s}.(\sum m_i \bar{C}_i)$ stays constant. Since $r(X_s)$ divides
    $K_{X_s}. (\sum m_i \bar{C}_i)$, this gives that $r(X_s)\mid r$. We conclude that $r(X_s) = r(X_\eta)$ by
    Proposition \ref{prop:constancy_of_p_index}.
\end{proof}
 
 To conclude this section, we now show that the integral Hodge conjecture for $1$-cycles on Fano varieties implies that condition (3) in Proposition \ref{prop:index_almost_2} holds for all Fano varieties in characteristic zero (see Proposition \ref{prop:ihc} below). In particular, assuming the integral Hodge conjecture, it follows from Proposition \ref{prop:index_almost_2} and Proposition \ref{prop:ihc} that the index is constant in families of Fano varieties of mixed characteristic.
 
 \begin{proposition}\label{prop:ihc}
 	Assume the integral Hodge conjecture for $1$-cycles on Fano varieties over $\mathbb C$. Then, for an algebraically closed field $k$ of characteristic zero and a Fano variety $X$ over $k$, there exists an integral $1$-cycle $T$ on $X$ such that $-K_X.T = r(X)$.
 \end{proposition}
 \begin{proof} First, assume $k=\mathbb C$. 
 	Let $n:=\dim X$. Since $X$ is Fano, the Hodge decomposition and the fact
 	that $\mathrm{H}^i(X, \OO_X)=0$ for $i>0$ give $\mathrm{H}^{2n-2}(X, \CC) = \mathrm{H}^{n-1,n-1}(X)$. From Poincar\'e duality and the fact that the Picard group is torsion free, we have a unimodular pairing
 	$\mathrm{H}^2(X,\ZZ)\times \mathrm{H}^{2n-2}(X,\ZZ)\to\ZZ$ given by cup product. Write now $-K_X=rH$ for
 	$r=r(X)$ the index and define a homomorphism $\mathrm{H}^2(X,\ZZ)\to\ZZ$ sending $H\mapsto 1$ and extending by
 	zero elsewhere. Unimodularity of the cup product implies that there is a class $T\in \mathrm{H}^{2n-2}(X,\ZZ)$ so
 	that $(H,T)=1$. By the surjectivity of the cycle class map  $\mathrm{CH}_1(X) \to \mathrm{H}^{n-1,n-1}(X)\cap \mathrm{H}^{2n-2}(X,\ZZ)$ (which follows from our assumptions), the cohomology class $T$ can be represented by a cycle $\sum m_i C_i$, where
 	$m_i\in\ZZ$ and $C_i$ irreducible curves on $X$. This concludes the proof of the proposition for $k=\mathbb C$.
 	
 	To conclude the proof of the proposition, let $X$ be a Fano variety over an algebraically closed field $k$ of characteristic zero.   Let $K\subset k$ be an algebraically closed subfield  with an embedding $K\to \mathbb C$ and let   $Y\to \Spec K$ be a Fano variety over $K$  such that $Y_k \cong X$. Then, by what we have just shown, there is an integral $1$-cycle $T'$ on $Y_{\mathbb C}$ such that $-K_{Y_{\CC}}.T' = r(Y_{\CC}) = r(Y) = r(X)$.  We now use a standard specialization argument to conclude the proof.
 	
 	Let $K\subset L\subset \mathbb C$ be a   subfield of $\mathbb C$ which is finitely generated over $K$ such that the integral $1$-cycle $T'$ on $Y_{\mathbb C}$ descends to an integral $1$-cycle $T''$ on $Y_L$.  Moreover, let 
 $U$ be a quasi-projective integral scheme over $K$ whose function field  $K(U)$ is $L$,  and let $\mathcal{Y}\to U$ be a smooth proper morphism whose geometric fibres are Fano varieties such that   $\mathcal{Y}\times_U \Spec  L\cong Y_{L}$. Let $\mathcal{T}$ be the closure of $T''$ in $\mathcal Y$, and note that $\mathcal T$ extends    the integral $1$-cycle $T''$ on $Y_{L}$. Now, note that the generic fibres of $Y\times_K U\to U$ and $\mathcal Y\to U$ are   isomorphic over $L$. Therefore,  replacing $U$ by a dense open if necessary,   we have that $Y\times_K U$ is isomorphic to $\mathcal Y$ over $U$ (by ``spreading out'' of isomorphisms).  Let $u$ be  a closed point of $U$. The integral $1$-cycle $\mathcal T$ on $Y\times_K U$ restricts to an integral $1$-cycle $\mathcal T_u$ on $Y=Y\times_K \Spec k(u)$ such that $-K_{Y}.\mathcal T_u = r(Y)$. Define $T$ to be the $\mathcal T_u \times_K \Spec k$. Note that $T$ is an integral $1$-cycle on $X$ with  the sought property.
 \end{proof}
  
  \begin{corollary}
Let $S$ be a trait with generic point $\eta$ and closed point $s$. Let $X\to S$ be a smooth proper morphism of schemes whose geometric fibres are Fano varieties. If $\mathrm{char}(k(\eta)) =0$ and the integral Hodge conjecture holds for $1$-cycles on Fano varieties over $\mathbb C$, then $r(X_s) = r(X_\eta)$.
  \end{corollary}
\begin{proof}
	This follows from $(3)$ of  Proposition \ref{prop:index_almost_2} and Proposition \ref{prop:ihc}.
\end{proof}
\appendix

\section{Varieties with $\mathrm{CH}_0(X)=\ZZ$}
To show that the Brauer group of a rationally chain connected
variety is killed by some positive integer (Proposition \ref{prop:Brauer_is_killed}), one can also argue using the
decomposition of the diagonal of Bloch-Srinivas, as we show now.

\begin{theorem}\label{thm:appendix}
	Let $X$ be a smooth projective variety over an algebraically closed field $k$ of characteristic $p\geq0$. Assume that there is an algebraically closed field $\Omega$ of infinite transcendence degree over $k$ such that
	that $\mathrm{CH}_0(X_\Omega) = \ZZ$. Then there exists an integer $m \geq 1$ such that $\Br(X)$ is $m$-torsion. 
\end{theorem}
\begin{proof} 
        Let $n=\dim X$. 
        To prove the result, by a standard specialization argument, it suffices to show that $\Br(X_{\Omega})$ is killed by some integer $m\geq 1$.
       Thus, to prove the theorem, we may and do assume that $X= X_{\Omega}$.
       
       By the decomposition of the diagonal of Bloch and Srinivas \cite{blochsrinivas} there is an
        integer $m\geq 1$ such that in $\CH^n(X\times X)$ we have $$ m\Delta = X\times x + Z,$$ where $x\in X$ is any
        point and $Z\subset X\times X$ is an $n$-cycle whose projection under $\pr_1$ does not dominate $X$
        \cite[Thm.~3.10]{voisin}. In particular, there is a divisor $D$ on $X$ such that $Z$ is supported on $D\times
        X$. Now, start with a class $\beta\in \Br (X) = \mathrm{H}^2_\et(X,\GG_m)$. Under our assumptions, it will be a
        torsion cohomology class. Assume it is $d$-torsion, so that it is an element of $\Br(X)[d]$. From the Kummer
        sequence in the fppf topology (note here that if $(d,p)=1$ then we can work in the \'etale topology), we have a
        short exact sequence $$ 0\to \Pic(X)/d\Pic(X)\to \mathrm{H}^2_\fppf(X, \mu_d) \to \Br(X)[d] \to 0 $$ where we
        have used that $\mathrm{H}^1_\fppf(X, \GG_m) = \mathrm{H}^1_\et(X, \GG_m) = \Pic X$ and likewise for second
        cohomology of $\GG_m$ by a theorem of Grothendieck saying that fppf and \'etale cohomology agree for smooth
        groups. Let $\alpha\in \mathrm{H}^2_\fppf(X,\mu_d)$ be any class in the preimage of $\beta$. (If $(d,p)=1$, then
        the group $\mathrm{H}^2_\et(X,\mu_d)$ is finite, and hence so is $\Br(X)[d]$. However, $\mathrm{H}^2_\fppf(X,
        \mu_p)$ does not have to be finite.) Each element $\gamma$ in $\CH^n(X\times X)$ is a correspondence on $X$, and
        therefore induces a morphism $\gamma^*:\mathrm{H}^i_\fppf(X,\mu_d)\to \mathrm{H}^i_\fppf(X,\mu_d)$ given by
        $\gamma^*\alpha:=\pr_{2*}(\pr_1^*\alpha\cup [\gamma])$ where $[\gamma]\in \mathrm{H}^{2n}_\fppf(X\times
        X,\mu_d)$ is the image of $\gamma$ under the cycle class map $$\cl: \CH^n(X\times X)\to
        \mathrm{H}^{2n}_\fppf(X\times X,\mu_d);$$ see Remark \ref{rem:appendix} below. In particular, the diagonal
        induces the identity morphism on $\mathrm{H}^2_\fppf(X,\mu_d)$, whereas the class $X\times x$ is the zero map.
        We thus have $$ m\alpha = [Z]^*\alpha. $$ Let now $p,q$ be the projections of $Z$ onto $D$ and $X$ respectively.
        Hence, if $i:D\to X$ denotes the inclusion, then we have proven that multiplication by $m$ on
        $\mathrm{H}^2_\fppf(X,\mu_d)$ factors as follows:
	$$ 
	\xymatrix{
		& \mathrm{H}^0_\fppf(D, \mu_d)\ar[dr]^{i_*} & \\
		\mathrm{H}^2_\fppf(X,\mu_d)\ar[rr]^m \ar[d]\ar[ur]^{p_*q^*} & & \mathrm{H}^2_\fppf(X,\mu_d)\ar[d] \\
		\Br(X)[d]\ar[rr]^m & & \Br(X)[d].
	}
	$$
	The map $i_*$ maps a multiple of the fundamental class of $D$ (a divisor) to its Chern class, so in particular the
	image of $i_*$ is contained in the image of $\Pic(X)/d\Pic(X)$. Thus,  when projected down to $\Br(X)[d]$, the image of $i_*$ must be zero.
	In other words, every $d$-torsion class in the Brauer group of $X$ is killed by $m$. 
\end{proof} 

\begin{remark}
	If $X$ is a non-supersingular K3 surface over an algebraic closure  $\bar{\mathbb F_p}$ of $\mathbb F_p$, then $\CH_0(X) = \ZZ$. (Let $A_0(X)$ be the Chow group of zero cycles of degree zero on $X$. By    \cite{MilneRojtman, Roitman}, the natural map $A_0(X) \to \mathrm{Alb}(X)(\bar{\mathbb F_p})$ induces an isomorphism on torsion subgroups. Thus, as the Albanese of $X$ is trivial,  the group $A_0(X)$ is torsion-free. However, as $A_0(X)$ is  a torsion abelian group, we  conclude that $A_0(X) =0$ and thus $\mathrm{CH}_0(X) =\ZZ$.)    Now, since $X$ is non-supersingular, there is no integer $d\geq 1$ such that the Brauer group of $X$ is killed by $d$.    
\end{remark}

\begin{remark}\label{rem:appendix}
        Let $d$ be a positive integer. Note that  the Kummer sequence $0\to \mu_d\to\GG_m\to\GG_m\to0$ is exact in the fppf topology. In
        particular, every class in $\Br(X)[d]$ comes from a class in $\mathrm{H}^2_\fppf(X,\mu_d)$. Now, even though the latter group
        does not have to be finite, we do have cycle class maps in the fppf topology. We could not find an explicit
        reference for this, but taking a locally free resolution of the ideal sheaf of a subvariety, using the existence
        of a $c_1$ map for line bundles in the fppf topology and extending by linearity and cup product gives Chern
        classes and hence a cycle class map into $\mathrm{H}^{2i}_{\mathrm{fppf}}(X, \mu_d)$. 
\end{remark}

\begin{remark} Note that the assumption in Theorem \ref{thm:appendix} holds for a rationally chain connected smooth projective variety $X$ over $k$.\end{remark}

\begin{remark}  The fact that the prime-to-$p$ part of the Brauer group of a rationally chain connected variety $X$ is killed by some integer $d\geq 1$ can also be deduced from   work of Colliot-Th\'el\`ene \cite{CTinv} (see also \cite[Theorem 1.4 and Lemma 1.7]{Auel}).
\end{remark}

\begin{remark}[Colliot-Th\'el\`ene]
	Salberger has proven (unpublished) the following generalization of Theorem \ref{thm:appendix}. Let $X$ be a smooth projective variety over an algebraically closed field $k$. Assume that there is a curve $C$ over $k$ and a morphism $C\to X$ such that, for any algebraically closed field $\Omega$ containing $k$, the induced morphism $\CH_0(C_\Omega)\to \CH_0(X_\Omega)$ is surjective. Then, there exists an integer $m\geq 1$ such that $\Br(X)$ is $m$-torsion.
\end{remark}

\bibliography{refs}{}
\bibliographystyle{plain}

\end{document}